\documentclass[a4paper,12pt,final]{amsart}
\usepackage{times,a4wide,mathrsfs,amssymb,amsmath,amsthm,enumerate,xypic,tikzsymbols,dsfont}
\usepackage{amscd}
\usepackage[all]{xy}
\usepackage{showkeys}
\usepackage{nameref}

\usepackage{amssymb}
\usepackage{graphicx}

\usepackage{enumitem}

\usepackage{xcolor}

\newcommand{\Gr}{{\operatorname{Gr}}}

\renewcommand{\Im}{{\operatorname{Im}}}

\newcommand{\CH}{\operatorname{CH}}

\newcommand{\Spec}{\operatorname{Spec}}

\renewcommand{\dim}{\operatorname{dim}}

\renewcommand{\Im}{\operatorname{Im}}
\newcommand{\wt}{\widetilde}

\newcommand{\C}{\mathbb{C}}

\renewcommand{\P}{\mathbb{P}}

\newcommand{\Q}{\mathbb{Q}}

 \newcommand{\sC}{\mathcal{C}}

\newcommand{\sM}{\mathcal{M}}
\newcommand{\JJ}{\mathcal{J}}

\newcommand{\sO}{\mathcal{O}}

\newcommand{\sB}{\mathcal{B}}

 \newcommand{\sF}{\mathcal{F}}

\newcommand{\sX}{\mathcal{X}}
\newcommand{\sY}{\mathcal{Y}}

\newcommand{\ima}{\hbox{Im}}
\newcommand{\id}{\hbox{id}}

\newcommand{\one}{\mathds{1}}

 \numberwithin{equation}{section}

\theoremstyle{plain}
\newtheorem{thm}[equation]{Theorem}

\newtheorem{prop}[equation]{Proposition}
\newtheorem{lm}[equation]{Lemma}
\newtheorem{cor}[equation]{Corollary}
\newtheorem{conj}[equation]{Conjecture}

\newtheorem{sublm}[equation]{Sublemma}

\newtheorem{convention}{Conventions}

\DeclareMathOperator{\GDCH}{GDCH}

\theoremstyle{definition}
\newtheorem{defn}[equation]{Definition}

\newtheorem{rk}[equation]{Remark}

\newtheorem{nonumbering}{Theorem}

\newtheorem{nonumberingc}{Corollary}

\thanks{M.B. and R.L. are supported by ANR grant ANR-20-CE40-0023. M.B. has benefitted of the support of the \it Dotation Exceptionnelle 2022 \rm of the Université de Montpellier} 

\newtheorem{nonumberingt}{Acknowledgments}

\newif\ifHideFoot
\HideFoottrue  
\HideFootfalse  

\ifHideFoot

\newcommand{\Michele}[1]{}
\newcommand{\Robert}[1]{}

\else 

\newcommand{\marg}[1]{\normalsize{{
			\color{red}\footnote{{\color{blue}#1}}}{\marginpar[\vskip
			-.25cm{\color{red}\hfill$\Rightarrow$\tiny\thefootnote}]{\vskip
				-.2cm{\color{red}$\Leftarrow$\tiny\thefootnote}}}}}
\newcommand{\Michele}[1]{\marg{(Michele) #1}}
\newcommand{\Robert}[1]{\marg{(Robert) #1}}

\fi

\begin{document}

\title{On the Chow ring of Fano fourfolds of K3 type}

\author[M. Bolognesi]{Michele Bolognesi}
\address{Institut Montpellierain Alexander Grothendieck \\ %
Universit\'e de Montpellier \\ %
CNRS \\ %
Case Courrier 051 - Place Eug\`ene Bataillon \\ %
34095 Montpellier Cedex 5 \\ %
France}
\email{michele.bolognesi@umontpellier.fr}

\author[R. Laterveer]{Robert Laterveer}
\address{Institut de Recherche MathÃ©matique Avanc\'ee, CNRS \\ %
Universit\'e de Strasbourg, \\ %
 7 Rue Ren\'e Descartes,\\ %
 67084 Strasbourg Cedex, France.}
\email{robert.laterveer@math.unistra.fr}
 
\dedicatory{To the memory of Alberto Collino, with friendship and gratitude}

\begin{abstract} 
We show that a wide range of Fano varieties of K3 type, recently constructed by Bernardara, Fatighenti, Manivel and Tanturri in \cite{B+}, have a multiplicative Chow-K\"unneth decomposition, in the sense of Shen--Vial. It follows that the Chow ring of these Fano varieties behaves like that of K3 surfaces.
As a side result, we obtain some criteria for the Franchetta property of blown-up projective varieties.

\end{abstract}
\maketitle

\section{Introduction}

Given a smooth projective variety $X$ over $\C$, let $\CH^i(X):=CH^i(X)_{\Q}$ denote the Chow groups of $X$ (i.e. the groups of codimension $i$ algebraic cycles on $X$ with $\Q$-coefficients, modulo rational equivalence). The intersection product defines a ring structure on $\CH^\ast(X)=\bigoplus_i \CH^i(X)$, the {\em Chow ring\/} of $X$ \cite{F}. 

In the special case of K3 surfaces, this ring structure has a remarkable property:

\begin{thm}[Beauville--Voisin \cite{BV}]\label{bv} Let $S$ be a projective K3 surface. 
The $\Q$-subalgebra
  \[   \bigl\langle  \CH^1(S), c_j(S) \bigr\rangle\ \ \ \subset\ \CH^\ast(S) \]
  injects into cohomology under the cycle class map.
  \end{thm}

This remarkable behaviour of K3 surfaces has led to Beauville's ``splitting property conjecture'' \cite{Beau3}, and to the concept of {\em multiplicative Chow--K\"unneth decomposition\/} \cite{SV}. In short, a multiplicative Chow--K\"unneth decomposition (which we will abbreviate to {\em MCK decomposition\/}) is a graded decomposition of the motive of a variety that is compatible with intersection product (cf. Section \ref{mckdeco} below for details). Inspecting the proof of Theorem \ref{bv}, one can deduce that K3 surfaces have an MCK decomposition, which gives a nice, conceptual explanation for Theorem \ref{bv}.

The following conjecture has been formulated in \cite[Conjecture 1.4]{44} (cf. also \cite[Section 1.2]{FLV2}):

\begin{conj}\label{conj} Let $X$ be a Fano variety of K3 type (i.e.\footnote{There are other possible definitions that are less restrictive, as in \cite{B+} where some Fano fourfolds have $H^3(X,\Q)\not=0$. However, the Fano fourfolds considered in this paper are of K3 type in the narrow sense of our definition.} $X$ has dimension $2d$ and the Hodge numbers of $X$ verify $h^{p,q}(X)=0$ for all $p\not=q$ except for $h^{d-1,d+1}(X)=h^{d+1,d-1}(X)=1$).
Then $X$ has an MCK decomposition.
\end{conj}

The raison d'\^etre for this conjecture is that Fano varieties of K3 type are expected to be related to hyperk\"ahler varieties, and hyperk\"ahler varieties are
expected to have an MCK decomposition \cite{SV}.
Conjecture \ref{conj} has been verified in some cases: cubic fourfolds \cite{Diaz}, \cite{FLV2}, certain K\"uchle fourfolds \cite{37}, \cite{FLV2}, certain varieties on the Fatighenti--Mongardi list \cite{40}, \cite{41}, \cite{44}, \cite{59}.

Recently, Bernardara, Fatighenti, Manivel and Tanturri \cite{B+} have constructed many new Fano fourfolds of K3 type, by considering sections of vector bundles on products of homogeneous varieties.
The goal of this paper is to verify Conjecture \ref{conj} for the Fano fourfolds of \cite{B+}. 
The following is our main result:

\begin{nonumbering}[=Theorem \ref{main}] Let $X$ be one of the Fano fourfolds in Table 1. Then $X$ has an MCK decomposition.
\end{nonumbering}

As a concrete consequence of Theorem \ref{main}, the Chow ring of these Fano varieties shows a behaviour similar to that of K3 surfaces:

\begin{nonumberingc}[=Corollary \ref{mainc}] Let $X$ be one of the Fano fourfolds in Table 1. Then the image of the intersection product map
  \[ \CH^1(X)\otimes \CH^2(X)\ \to\  \CH^3(X) \]
  injects into cohomology. In other words, there are $\rho:=\dim H^2(X,\Q)$ distinguished 1-cycles $\ell_1,\ldots,\ell_\rho$ such that
  \[  \Im\Bigl(  \CH^1(X)\otimes \CH^2(X)\ \to\  \CH^3(X)\Bigr) =\bigoplus_{i=1}^\rho \Q[\ell_i]\ .\]
  \end{nonumberingc}

 To prove Theorem \ref{main}, we apply a general criterion dealing with MCK decompositions and blow-ups (Proposition \ref{bucrit}). In order to check that the hypotheses of this general criterion are verified, it suffices that certain families verify the {\em Franchetta property\/} (this is the content of Proposition \ref{busurface}). To establish the required instances of the Franchetta property, we were led to develop certain new techniques that might hold some independent interest (cf. Lemma \ref{bufrank1} and Proposition \ref{frk3}).

\newcommand{\sQ}{\mathcal{Q}}
\newcommand{\sU}{\mathcal{U}}
\newcommand{\sR}{\mathcal{R}}

 \vskip0.4cm

\begin{convention} In this paper, the word {\sl variety\/}  refers to a reduced irreducible scheme of finite type over $\C$. A {\sl subvariety\/} is a (possibly reducible) reduced subscheme which is equidimensional. All cycle class groups will be with rational coefficients. We write $\CH^i_{hom}(X)\subset \CH^i(X)$ for the subgroup of homologically trivial cycles.
A roman $X$ will indicate one variety over $\Spec \C$, while an italic $\sX$ will mean a family over a different scheme.

The contravariant category of Chow motives (i.e., pure motives with respect to rational equivalence as in \cite{Sc}, \cite{MNP}) will be denoted 
$\sM_{\rm rat}$.
\end{convention}

\vskip0.8cm

\begin{table}
\centering
\begin{tabular}{ ||c  c||} 
 \hline
 Fano fourfold: &  Ambient space $M$, vector bundle $F$: \\ 
 [0.5ex]
 \hline\hline
 C-1 &  $\P^1\times \P^5,\ \sO (0,3) \oplus \sO(0,1)$ \\
  \hline
 C-3 &  $\P^2 \times \P^5,\ \sO(1,2)\oplus \sQ_{\P^2}(0,1)$ \\ 
  \hline
 C-6 &  $\P^4 \times \P^5,\ \sQ_{\P^4}(0,1)\oplus \sO(2,1)$ \\
 \hline 
 C-10 &  $\P^1\times\P^1\times\P^5,\ \sO(0,0,3)\oplus \sO(0,1,1)\oplus \sO(1,0,1)$\\
 \hline 
 C-12 &  $ \P^1\times \P^2\times \P^5,\ \sO(0,1,2)\oplus \sQ_{\P^2}(0,0,1) \oplus \sO(1,0,1)$\\
 \hline
 C-15 &  $\P^2_1 \times \P^2_2 \times \P^5,\ \sQ_{\P^2_1}(0,0,1)\oplus \sQ_{\P^2_2}(0,0,1)\oplus \sO(1,1,1)$ \\
 \hline
 K3-24 &  $\P^1 \times \Gr(2,5),\ \sO(1,1)\oplus \sU^*_{\Gr(2,5)}(0,1)$ \\
 \hline
 K3-25 &  $\P^1 \times \Gr(2,6),\ \sO(1,1)\oplus \sO(0,1)^{\oplus 4}$ \\
 \hline
 K3-26 &  $\P^1 \times \Gr(3,6),\ \sO(1,1)\oplus \sO(0,1)^{\oplus 2}\oplus \bigwedge \sU^*_{\Gr(3,6)}$ \\
 \hline
 K3-28 &  $\P^1 \times \Gr(2,7),\ \sO(0,1)\oplus \sO(1,1) \oplus \sQ^*_{\Gr(2,7)}(0,1)$ \\
 \hline 
 K3-30 &  $\P^1 \times \Gr(2,4),\ \sO(1,2)$ \\
 \hline
 K3-40 &  $\P^1_1\times \P^1_2 \times \P^5,\ \sO(0,0,2)\oplus \sO(0,1,1) \oplus \sO(1,0,2)$\\
 \hline
 K3-41 &  $\P^1\times \P^1\times \P^3,\ \sO(1,1,2)$\\
 \hline
 K3-46 &  $\P^1\times \P^3\times \P^5,\ \sQ_{\P^3}(0,0,1)\oplus \sO(0,1,1)\oplus \sO(1,1,1)$\\
 \hline
 K3-47 &  $\P^1\times \P^3\times \P^3,\ \sO(0,1,1)^{\oplus 2}\oplus \sO(1,1,1)$\\
 \hline
 K3-49 &  $\P^1\times \P^2 \times \P^4,\ \sO(0,0,2)\oplus \sO(0,1,1) \oplus \sO(1,1,1)$ \\
 \hline
 K3-55 &  $\P^1\times \P^4\times \P^5,\ \sQ_{\P^4}(0,0,1)\oplus \sO(0,2,0) \oplus \sO(1,1,1)$ \\
 \hline
 K3-56 &  $\P^1 \times \P^1 \times \P^2\times \P^3,\ \sQ_{\P^2}(0,0,0,1)\oplus \sO(1,1,1,1)$ \\
 \hline
 K3-58 &  $\P^1\times \P^1 \times \P^1 \times \P^3,\ \sO(0,0,1,1)\oplus \sO(1,1,0,2)$ \\
 \hline
 K3-59 &  $\P^1\times\P^1\times \P^2\times \P^2,\ \sO(0,0,1,1)\oplus\sO(1,1,1,1)$ \\
 \hline
 K3-60 &  $(\P^1)^5,\ \sO(1,1,1,1,1)$\\
 [1ex]
 \hline

\end{tabular}
\caption{Fano fourfolds of K3 type, with their labelling as in \cite{B+}. The notation $M, F$ indicates that the Fano fourfold $X$ is obtained as section of the vector bundle $F$ on the variety $M$.}
\label{table:1}
\end{table}

\section{Preliminaries}

\subsection{MCK decomposition}\label{mckdeco}

\begin{defn}[Murre \cite{Mur}] Let $X$ be a smooth projective $n$-dimensional variety.. We say that $X$ has a {\em CK decomposition\/} if there exists a decomposition of the diagonal
   \[ \Delta_X= \pi^0_X+ \pi^1_X+\cdots +\pi^{2n}_X\ \ \ \hbox{in}\ \CH^n(X\times X)\ ,\]
  such that:
  
\begin{enumerate}
  
\item  the cycles $\pi^i_X$ are idempotents and $\pi_X^i\circ \pi_X^j=0$ if $i\neq j$;
\item $(\pi^i_X)_*H^*(X,\Q)=H^i(X,\Q)$.
  
\end{enumerate}  
  
  (NB: ``CK decomposition'' stands for ``Chow--K\"unneth decomposition''.)
\end{defn}

\begin{rk} According to Murre's conjectures \cite{Mur}, \cite{J4}, all smooth projective varieties should have a CK decomposition.
\end{rk}

\begin{defn}[Shen--Vial \cite{SV}] Let $X$ be a smooth projective variety of dimension $n$. Let $\Delta_X^{sm}\in \CH^{2n}(X\times X\times X)$ be the class of the small diagonal
  \[ \Delta_X^{sm}:=\bigl\{ (x,x,x)\ \vert\ x\in X\bigr\}\ \subset\ X\times X\times X\ .\]
  An {\em MCK decomposition\/} is a CK decomposition $\{\pi^i_X\}$ of $X$ that is {\em multiplicative\/}. That is, it satisfies
  
  \begin{equation}\label{vani} \pi^k_X\circ \Delta_X^{sm}\circ (\pi^i_X\times \pi^j_X)=0\ \ \ \hbox{in}\ \CH^{2n}(X\times X\times X)\ \ \ \hbox{for\ all\ }i+j\not=k,\end{equation}
 where $\pi_X^i\times \pi^j_X$ is by definition $(p_{13})^\ast(\pi^i_X)\cdot (p_{24})^\ast(\pi^j_X)\in \CH^{2n}(X^4)$, and $p_{rs}\colon X^4\to X^2$ is the projection on $r$th and $s$th factors.
  
 (NB: ``MCK decomposition'' stands for ``multiplicative Chow--K\"unneth decomposition''.) 
 
\end{defn} 
 
  
  \begin{rk}\label{rem:mck} It is not hard to see that the vanishing \eqref{vani} is always true modulo homological equivalence. This is due to the standard fact that the cup product in cohomology respects the grading.
  \end{rk}
  
    The small diagonal (considered as a correspondence from $X\times X$ to $X$) induces the {\em multiplication morphism\/}
    \[ \Delta_X^{sm}\colon\ \  h(X)\otimes h(X)\ \to\ h(X)\ \ \ \hbox{in}\ \sM_{\rm rat}\ .\]
 Now, suppose that $X$ has a CK decomposition as follows:
  \[ h(X)=\bigoplus_{i=0}^{2n} h^i(X)\ \ \ \hbox{in}\ \sM_{\rm rat}\ .\]
 This decomposition is by definition multiplicative if for any $i,j$ the composition
  \[ h^i(X)\otimes h^j(X)\ \to\ h(X)\otimes h(X)\ \xrightarrow{\Delta_X^{sm}}\ h(X)\ \ \ \hbox{in}\ \sM_{\rm rat}\]
  factors through $h^{i+j}(X)$.
  
If we assume that $X$ has an MCK decomposition, then, by setting
    \[ \CH^i_{(j)}(X):= (\pi_X^{2i-j})_\ast \CH^i(X) \ ,\]
    one obtains a bigraded ring structure on the Chow ring. That is, the intersection product sends $\CH^i_{(j)}(X)\otimes \CH^{i^\prime}_{(j^\prime)}(X) $ to  $\CH^{i+i^\prime}_{(j+j^\prime)}(X)$.
    
      It is a reasonable expectation that for any $X$ with an MCK decomposition, one has
    \[ \CH^i_{(j)}(X)\stackrel{??}{=}0\ \ \ \hbox{for}\ j<0\ ,\ \ \ \CH^i_{(0)}(X)\cap \CH^i_{hom}(X)\stackrel{??}{=}0\ .\]
    This is strictly related to Murre's conjectures B and D, that were formulated for any CK decomposition \cite{Mur}. While these questions are open in general, here is a partial result:
    
    \begin{lm}\label{injection}
Let $X$ be a smooth projective variety of dimension $n$ with an MCK decomposition $\{\pi^j_X\}$. 
Assume $H^3(X,\Q)=0$, and assume also the following support condition is satisfied:
  \begin{equation}\label{suppcond}  \pi^0_X \ \hbox{and}\ \pi^2_X\ \hbox{are\ supported\ on}\ x\times X\ \hbox{resp.}\ S\times X\ ,\end{equation}
  where $x\in X$ is a point and $S\subset X$ is a surface.
  
Then $\CH^i_{(0)}(X)$ injects into cohomology via the cycle class map for $i\geq n-1$.
\end{lm}

\begin{proof} By self-duality, the support condition implies that $\pi^{2n}_X$ and $\pi^{2n-2}_X$ are supported on $X\times x$, resp. on $X\times S$. 

By definition, $\pi^{2n}_X$ acts as the identity on $\CH^n_{(0)}(X)$. On the other hand, the action of $\pi^{2n}_X$ on $\CH^n_{}(X)$ factors over $\CH^0(x)=\Q$ and so 
  \[    (\pi^{2n}_X)_\ast=0\ \colon\ \ \CH^n_{hom}(X)\ \to\ \CH^n(X)\ .\]
  It follows that
    \[  \CH^n_{(0)}(X)   \cap   \CH^n_{hom}(X)=0\ \]
    as claimed.
    
 The argument for $i=n-1$ is similar. First of all, the assumption on $H^3(X,\Q)$ implies that $\CH^{n-1}_{hom}(X)=\CH^{n-1}_{AJ}(X)$. Now
  the projector $\pi^{2n-2}_X$ acts as the identity on $\CH^{n-1}_{(0)}(X)$. On the other hand, the action of $\pi^{2n-2}_X$ on $\CH^{n-1}_{hom}(X)=\CH^{n-1}_{AJ}(X)$ factors over $\CH^1_{AJ}(S)=0$ and so 
  \[    (\pi^{2n-2}_X)_\ast=0\ \colon\ \ \CH^{n-1}_{hom}(X)\ \to\ \CH^{n-1}(X)\ .\]
  It follows that
    \[  \CH^{n-1}_{(0)}(X)   \cap   \CH^{n-1}_{hom}(X)=0\ \]
    as claimed.
\end{proof}

\begin{rk} The support condition \eqref{suppcond} is verified by all MCK decompositions that have been constructed. In general, this condition is verified provided there are no ``phantom motives'' (i.e. Chow motives whose cohomological realization is zero); in particular, condition \eqref{suppcond} holds for $X$ when $X$ is Kimura finite-dimensional.
\end{rk}

\begin{thm}\label{mckcubic}\cite{Diaz,FLV2}
All smooth cubic hypersurfaces of any dimension have an MCK decomposition.
\end{thm}

   \begin{rk}   
  Having an MCK decomposition is a very restrictive property, and it is related to Beauville's ``splitting property conjecture'' \cite{Beau3}. 
  Let us just produce a short list of examples: hyperelliptic curves have an MCK decomposition \cite[Example 8.16]{SV}, but the very general curve of genus $\ge 3$ does not \cite[Example 2.3]{FLV2}. In dimension two, a smooth quartic surface in $\P^3$ has an MCK decomposition, but a very general surface of degree $ \ge 7$ in $\P^3$ should not have one \cite[Proposition 3.4]{FLV2}.
For more details, and other examples of varieties with an MCK decomposition, the reader may check \cite[Section 8]{SV}, as well as \cite{V6}, \cite{SV2}, \cite{FTV}, \cite{37}, \cite{38}, \cite{39}, \cite{40}, \cite{44}, \cite{45}, \cite{46}, \cite{48}, \cite{55}, \cite{FLV2}, \cite{g8}, \cite{59}, \cite{60}.
   \end{rk}

\subsection{Franchetta property}

The Franchetta property is often strictly related to MCK decompositions (see for example \cite{FLV2}).

\begin{defn}\label{frank} Let $\pi\colon\sY\to \sB$ be a smooth projective morphism, where $\sY, \sB$ are smooth quasi-projective varieties. For any $b\in\sB$, we write $Y_b$ for the fiber $\pi^{-1}(b)$.
One says that $\sY\to \sB$ has the {\em Franchetta property in codimension $j$\/} if the following property holds:

\smallskip

 for every $\Gamma\in \CH^j(\sY)$ such that the restriction $\Gamma\vert_{Y_b}$ to the fiber is homologically trivial for all $b\in \sB$, the restriction $\Gamma\vert_{Y_b}$ is zero in $\CH^j(Y_b)$ for all $b\in \sB$.
 
 We will say that $\sY\to \sB$ has the {\em Franchetta property\/} if $\sY\to \sB$ has the Franchetta property in codimension $j$ for all $j$.
 \end{defn}
 
 This property has been studied in \cite{BL}, \cite{FLV}, \cite{FLV2}, \cite{FLV3}, \cite{BL22}.
 
 \begin{defn}\label{def:gd} Given a family $\sY\to \sB$ with the properties described above, and given $Y:=Y_b$ a fiber, we write
   \[ \GDCH^j_\sB(Y):=\Im\Bigl( 
  \CH^j(\sY)\to \CH^j(Y)\Bigr) \]
   for the subgroup of {\em generically defined cycles}. 
  Whenever it is clear to which family we are referring, we will often suppress the index $\sB$ from the notation.
  \end{defn}
  
 We remark that, with this notation, the Franchetta property is equivalent to saying that $\GDCH^\ast_\sB(Y)$ injects into cohomology, under the cycle class map.
 
\smallskip 
 
The Franchetta property for a family $\sX \to B$ does not imply and is not implied by the Franchetta property for a subfamily $\sX^\prime \to B^\prime$, where $B^\prime$ is a closed subscheme of $B$. If $B^\prime \to B$ is a dominant morphism, the Franchetta property for the base-changed family $\sX_{B^\prime} \to B^\prime$ implies the property for $\sX \to B$ \cite[Remark 2.6]{FLV}.

\smallskip 
 
  The next lemma describes the behaviour of the Franchetta property under blow-ups.
 
\begin{lm}\label{bufrank}
Let $\sX\to B$ be a family of smooth projective varieties (with fiber $X$), and $\sY\to B$ a family of smooth subvarieties. Let $\tilde{\sX}\to B$ denote the blow-up of $\sX$ along $\sY$.  Then $\sX\to B$ and $\sY\to B$ have the Franchetta property if and only if $\tilde{\sX}$ has the Franchetta property with respect to $B$.
\end{lm}

\begin{proof}
(IF direction) Let $E\subset \tilde{X}$ be the exceptional divisor of $Y$. Suppose that we have a family of diagrams over a scheme ${B}$:

\[
   \xymatrix{ \mathcal{E}\ar@{^{(}->}[r]\ar[d]   & \tilde{\sX} \ar[d] \\
\mathcal{Y}\ar@{^{(}->}[r]\ar[dr] & \sX\ar[d] \\
 & {B}, }  
 \]

where $\mathcal{E}$ is the exceptional divisor over $\sY$. That is, the varieties $X,\ Y,\ E$ and $\tilde{X}$ all deform along the same base scheme.\\

 The blow-up construction gives rise to a "Mayer-Vietoris" short exact sequence both in cohomology and in the Chow ring. If we denote by $cl: \mathrm{CH}^*(X)\to \mathrm{H}^*(X)$ the cycle class map, then we have a commutative diagram:

\[
\xymatrix{
0 \ar[r] & \CH^*(X)\ar[r]^(.4){t^{CH}}\ar[d]^{cl} & \CH^*(\tilde{X}) \oplus \CH^*(Y)\ar[d]^{cl}\ar[r]^(.7){v^{CH}}   & \CH^*(E)\ar[d]^{cl}\ar[r] & 0 \\
0 \ar[r] & \rm H^*(X)\ar[r]^(.4){t^{H}} & \rm H^*(\tilde{X}) \oplus \rm H^*(Y)\ar[r]^(.6){v^{H}}  & \rm H^*(E)\ar[r] & 0, } 
\]
where the first map is just pull-back, and the second is the difference of pullbacks. There exists an analogous commutative diagram on the level of the families $\sX$, $\sY$, $\wt{\sX}$.
The Franchetta property for $\tilde{X}$ is equivalent to $\GDCH^\ast_B(\tilde{X})$ injecting in cohomology. Consider now $\alpha \in \GDCH^\ast_B(X)$ such that $cl(\alpha)=0$, hence $t^H(cl(\alpha))=(0,0)\in H^*(\tilde{X})\oplus H^*(Y)$. We observe that $t^{CH}(\alpha)$ is also generically defined. By Franchetta for $\tilde{X}$ this means that $t^{CH}(\alpha)$ must be of the form $(0,\omega)$, for $\omega\in \rm A^*(Y)$. But the only cycle of this type in the image of $t^{CH}$ (which is the pullback) is $(0,0)$, and by injectivity of $t^{CH}$, we have $\alpha=0$. Hence the Franchetta property holds for $X$. Since the above exact sequences are all split, the Franchetta property for $E$ is implied by the Franchetta property for $\wt{X}$. The projective bundle formula then implies the same property for $Y$.

\medskip

(ONLY IF direction) We assume that both $X$ and $Y$ have the Franchetta property. First of all it is not hard to observe that the Franchetta property holds for $Y$ if and only if if holds for its exceptional divisor $E$. Now suppose that we have a generically defined cycle $\gamma \in \rm A^*(\tilde{X})$ such that $cl(\gamma)=0$. The cycle $(\gamma, 0)\in \rm A^*(\tilde{X})\oplus \rm A^*(Y)$ is sent to $(0,0)$ by the cycle class map. We remark that $v^{CH}(\gamma,0)$ is also generically defined. Since $v^H(cl(\gamma,0))=0$, by the Franchetta property of $E$, we have that $v^{CH}(\gamma,0)=0$. By the exactness of the sequence this means that $(\gamma,0)=t^{CH}(\sigma)$, for some $\sigma \in \rm A^*(X)$, which is certainly generically defined (since the exact sequences are split). But $t^H(cl(\sigma))=(0,0)$ and $t^H$ is injective, hence by the Franchetta property for $X$ we find $\sigma=0$ and $\gamma=0$ as well.
\end{proof}

We stress the fact that in Lemma \ref{bufrank}, it is essential to consider the Franchetta property {\em with respect to a common base $B$} for the 3 varieties $\wt{X}$, $X$ and $Y$.
This fact makes Lemma \ref{bufrank} difficult to apply in practice. For example, in many cases in \cite{B+} the variety $\wt{X}$ is obtained as section of some vector bundle $F$ on a product $M=M_1\times M_2$, and the projection $M\to M_1$ exhibits $\wt{X}$ as the blow-up of a cubic fourfold $X$ with center a rational surface $Y\subset X$. In this setting, the common base
$B$ is given by sections of $F$ on $M$, and it is not a priori clear that the cubic fourfolds $X$ verify Franchetta with respect to this base $B$.

To deal with this complication, we were led to develop some technical results (Lemma \ref{bufrank1} and Proposition \ref{frk3} below), that are tailor-made to fit the constructions of \cite{B+}.

%
%
%

\subsection{The Franchetta property for K3 surfaces}

In this section we collect known results for K3 surfaces, that will be useful in the rest of the paper. Thanks to Pavic-Shen-Yin \cite{PSY} and Fu and the second named author \cite{FL}, we have the following:

\begin{thm}
The Franchetta property holds for K3 surfaces of genus $g$ with $2\leq g \leq 10$ and for $g=12,13,14,16,18,20$.
\end{thm}

\begin{rk} Let $S$ be a K3 surface considered in \cite{PSY} (i.e. the genus $g$ of $S$ is in between 2 and 10, or $g=12,13,16,18,20$), embedded in some ambient space $W_g$ (these are the ambient spaces of the Mukai model of the K3 surface \cite{M1}, \cite{M2}, \cite{M3}, \cite{M4}, cf. the list in \cite{PSY}). Let moreover $\sF_g$ be the moduli stack of these K3 surfaces. The idea in \cite{PSY} consists in a two-step argument. First, through a projective bundle argument, they show that there is an injection
  \[  \GDCH^2_{\sF_g}(S)\ \hookrightarrow\ \ima \bigl( \CH^2(W_g)\to \CH^2(S)\bigr)\ .\]
  Secondly, they show that there is an injection
  \begin{equation}\label{injpsy}  
\ima \bigl( \CH^2(W_g)\to \CH^2(S)\bigr)\ \hookrightarrow\ H^4(S,\Q)\ ,\end{equation}
given by the cycle class map.
\end{rk}

\subsection{Varieties with trivial Chow groups} The following is well-known; we recall it for convenience:

\begin{lm}\label{triv} Let $M$ be a smooth projective variety. The following are equivalent:

\noindent
(i) The motive of $M$ is of Tate type: 
    \[ h(M)\cong \bigoplus \one(\ast)\ \ \ \hbox{in}\ \sM_{\rm rat}\ ;\]

\noindent
(ii) The cycle class map induces an isomorphism $\CH^\ast(M)\cong H^\ast(M,\Q)$\,;

\noindent
(iii) $\CH^\ast_{\rm hom}(M)=0$\,;

\noindent
(iv) $\CH^\ast(M)$ is a finite-dimensional $\Q$-vector space;

\noindent
(v) The natural map $\CH^\ast(M)\otimes \CH^\ast(M)\to \CH^\ast(M\times M)$ is an isomorphism.
\end{lm}

\begin{proof} The implications (i)$\Rightarrow$(ii)$\Rightarrow$(iii)$\Rightarrow$(iv) are obvious. The implication (iv)$\Rightarrow$(i) is \cite{Kim2} or \cite{V10}. The implication (i)$\Rightarrow$(5) follows readily from the fact that $h(M\times M)=h(M)\otimes h(M)$ and $\one(\ell)\otimes\one(m)=\one(\ell+m)$. Finally, to see that (v)$\Rightarrow$(iv), one notes that (v) implies the decomposition of the diagonal
  \[ \Delta_M=\sum_{j=1}^r \alpha_j\times \beta_j\ \ \ \hbox{in}\ \CH^{\dim M}(M\times M)\ ,\]
  where $\alpha_j,\beta_j\in \CH^\ast(M)$.
 Letting this decomposition act on $\CH^\ast(M)$, one finds that the identity factors over an $r$-dimensional $\Q$-vector space, and so (iv) holds.
 \end{proof}
 
 \begin{defn}[Voisin {\cite[Section 3.1]{V0}}] A smooth projective variety $M$ is said to have {\em trivial Chow groups\/} if $M$ verifies any of the equivalent conditions of Lemma \ref{triv}.
 \end{defn}

\begin{prop} Let $X$ be a smooth projective variety with trivial Chow groups. Then $X$ has an MCK decomposition, with the property that
  \[ \CH^\ast(X)= \CH^\ast_{(0)}(X)\ .\]
\end{prop}

\begin{proof} As remarked above, the required vanishing \eqref{vani} is always true in cohomology. Since $\CH^\ast_{hom}(X^3)=0$, the required vanishing \eqref{vani} also holds in $\CH^{2n}(X^3)$.
As for the second statement, we observe that
 \[  (\pi^j_X)_\ast \CH^i(X)\ \subset\ \CH^i_{hom}(X)=0\ \ \forall\ j\not=2i\ .\]
 \end{proof}

\subsection{Cayley trick}

An important tool for our proofs will be the well-known "Cayley trick". For the sake of self-containedness, we prefer to recall it here. For a projective variety $X$ with a vector bundle $E$, we will denote by $Z(X,E)$ the zero locus of a general section of $E$.

\begin{lm}\label{cayley}
Let $M_2$ be a smooth projective variety. Let $L$ and $F$ be a line bundle resp. a vector bundle on $M_2$, such that $\sO_{\P^1}(1)\boxtimes L \oplus \sO_{\P^1}\boxtimes F$ is globally generated on $\P^1\times M_2$, and consider 
  \[ \wt{X}=Z(\P^1\times M_2, \sO_{\P^1}(1)\boxtimes L \oplus \sO_{\P^1}\boxtimes F)\ .\] 
Then $\wt{X}$ is isomorphic to the blow-up $Bl_{Y}X$, where $X=Z(M_2,F)$ and $Y= Z(X,L^{\oplus 2})$.
\end{lm}

\begin{proof} This is (a special case of) \cite[Lemma 3.1]{B+}. Roughly speaking: if $(s,t)$ is a section of $\sO_{\P^1}(1)\boxtimes L \oplus \sO_{\P^1}\boxtimes F$ on the $\P^1\times M_2$, one applies \cite[Lemma 3.1]{B+} to the zero
locus of $s$ restricted to $\P^1 \times X(t)$, where $X(t)$ is the zero locus of the section $t$ on $M_2$.
\end{proof}

\section{Criteria}

The goal of this section is to establish some criteria, as broad as possible, to show that certain families of Fano varieties admit indeed a MCK decomposition. Certainly, one starting point is given by the following proposition, due to Shen and Vial \cite[Proposition 2.4]{SV}.

\smallskip

Let $X$ be a smooth projective variety and  $i: Y \hookrightarrow X$ a smooth subvariety of codimension $r+1$. We will denote by $\tilde{X}$ the blow-up of $X$ along $Y$. Before starting to prove our results we need to recall the following definition.

\begin{defn}
Let $X$ and $Y$ be smooth projective varieties. A correspondence $L\in \CH^p(X \times Y)$ is said to be of pure grade $j$ if $L\in \CH^p_{(j)}(X \times Y)$. In particular, a morphism $g : X \to Y$ is of pure grade 0 if its graph is in $\CH^d_{(0)} (X \times Y )$, where
$d = dim (Y)$.
\end{defn}

\begin{prop}\label{bucrit}
Assume that both $X$ and $Y$ admit multiplicative Chow-K\"unneth decompositions $\{\pi^i_X\}$ and $\{\pi^i_Y\}$, respectively, such that

\begin{enumerate}[label=(\roman*)]
\item the Chern classes of the normal bundle $N_{Y/X}$ sit in $\mathrm{CH}^*_{(0)}(Y)$;
\item the morphism $i:Y\to X$ is of pure grade 0.
\end{enumerate}

Then also $\tilde{X}$ admits a MCK decomposition.
\end{prop}

\begin{proof} This is \cite[Proposition 2.4]{SV2}. Note that in loc. cit., it is required that the MCK decompositions of $X$ and $Y$ are {\em self-dual\/}; however, this condition is actually always fulfilled \cite[Footnote 24]{FV}.

\end{proof}

For later use, we state a lemma:

\begin{lm}\label{injection2} Let $X,Y$ and $\tilde{X}$ be as in Proposition \ref{bucrit}. Assume $X$ and $Y$ verify the assumptions of Lemma \ref{injection}. Then
  \[ \CH^i_{(0)}(\tilde{X})\ \to\ H^{2i}(\tilde{X},\Q) \]
  is injective for $i\ge n-1$.
\end{lm}

\begin{proof} It is proven in \cite[Proposition 2.4]{SV2} that the natural isomorphisms
  \[  \CH^i(\tilde{X})\cong \CH^i(X)\oplus \bigoplus_{\ell=1}^r \CH^{i-\ell}(Y)\]
  are given by correspondences of pure grade 0 (and so they respect the second grading). In particular, this means that we have isomorphisms
    \[  \begin{split} &\CH^n_{(0)}(\tilde{X}) = \CH^n_{(0)}(X)\ ,\\
                             &\CH^{n-1}_{(0)}(\tilde{X}) = \CH^{n-1}_{(0)}(X)\oplus \CH^{n-1-r}_{(0)}(Y)\ .\\
                             \end{split}\]
  Since these isomorphisms are compatible with the cycle class map, and Lemma \ref{injection} applies to $X$ and to $Y$, this proves the claimed injectivity.                           
                             \end{proof}

We will now make use of Proposition \ref{bucrit} in order to craft a criterion more specific to the cases we are going to inspect.

\begin{prop}\label{busurface}
Let $\sX\to B$, $\sY\to B$, $\wt{\sX}\to B$ and
$X$, $Y$, $\wt{X}$ be as in Lemma \ref{bufrank}. 
Assume that both $X$ and $Y$ admit MCK decompositions $\pi^i_X$ resp. $\pi^i_Y$ that are generically defined (with respect to the base $B$), and that

\begin{enumerate}
\item either $\sX\to B$ has the Franchetta property and $Y$ has trivial Chow groups,
\item or $X$ has trivial Chow groups and $\sY\to B$ has the Franchetta property;
\end{enumerate}

Then the blown-up variety $\tilde{X}$ also admits an MCK decomposition (that is generically defined).
\end{prop}

\begin{proof}
As in the proof of Lemma \ref{bufrank}, the varieties $Y, X$ and $\tilde{X}$ are fibers of the families $\sY$, $\sX$, $\wt{\sX}$ over the common base $B$. Let
$n:=\dim X$.
 The proofs for $(1)$ and $(2)$ are similar, since in both cases we want to apply Proposition \ref{bucrit}. Let us start with $(1)$.

(1) The fact that $Y$ has trivial Chow groups implies that the Chern classes are injected into cohomology, and the condition $(i)$ of Proposition \ref{bucrit} holds true.


 In order to prove condition $(ii)$ of Proposition \ref{bucrit} we observe that, since $Y$ has trivial Chow groups, then the Chow ring of the product $X\times Y$ decomposes as a finite direct sum of copies of the Chow ring of 
 $X$, as follows: 
\begin{equation*}
 \mathrm{CH}^*(X\times Y) = \oplus_i \mathrm{CH}^*(X).
\end{equation*}  
The upshot is that the Franchetta property also holds for the product $X\times Y$. 


To prove that the graph of the inclusion $\iota\colon Y\to X$ is of pure grade 0, we need to prove the vanishing
  \[ \Lambda:= (\pi^i_X\times\pi^j_Y)_\ast (\Gamma_\iota)=0\ \ \hbox{in}\ \CH^n(X\times Y)\ \ \ \forall i+j\not=2n\ .\]
 To this end, we remark that the cycle $\Lambda$ is generically defined (with respect to $B$), and homologically trivial. The Franchetta property for 
$\sX\times_B \sY$ then implies that $\Lambda$ is rationally equivalent to 0.
This completes the proof of condition $(ii)$ of Proposition \ref{bucrit} in case $(1)$.

\smallskip

(2) 
Let us first check condition $(i)$ of Proposition \ref{bucrit}.
We need to ascertain the vanishing
  \[  (\pi^i_Y)_\ast c_k(N_{Y/X})=0\ \ \hbox{in}\ \CH^k(Y)\ \ \forall i\not=2k\ .\]
 To this end, we remark that the cycle $ (\pi^i_Y)_\ast c_k(N_{Y/X})$ is generically defined (with respect to $B$) and homologically trivial. The Franchetta property for $\sY$ then gives the required vanishing.

The argument for checking condition $(ii)$ is exactly the same as in $(1)$.
\end{proof}

Now that we have proven Proposition \ref{busurface}, in the next section we will see that it can be applied to a lot of examples of Fano fourfolds that were introduced in \cite{B+}.

\section{Fanos}

In this section we will consider the examples from the Bernardara--Fatighenti--Mongardi--Tanturri list for which one can apply Proposition \ref{busurface}. We will be applying either part 1 or 2 of the proposition. The goal of course is to exhibit new classes of Fano varieties of K3 type, that have 
an MCK decomposition. The following is the main result of this paper:

\begin{thm}\label{main} 
Let $X$ be one of the Fano fourfolds in Table 1. Then $X$ has an MCK decomposition.
\end{thm}

To prove Theorem \ref{main}, we separately treat each family of Table 1.

\subsection{MCK decomposition for six families of Fano fourfolds coming from cubic fourfolds}

In this subsection we introduce six families of smooth Fano fourfolds from \cite{B+}, that are obtained as blow-ups of cubic fourfolds, or of fourfolds strictly related to cubic fourfolds. For more information on the motive of a cubic fourfold and its relation with K3 surfaces see \cite{BP} \cite{Bu}.

\smallskip

To treat these families, we are going to use the following result:

\begin{lm}\label{bufrank1}

Let $M=M_1\times M_2$, and $F\to M$ a globally generated vector bundle.

Let 
  \[ \wt{\sX}\ \to\  B\subset \bar{B}:= \P H^0(M, F)\]
  be the family obtained as smooth dimensionally transverse sections of $F$. Assume that the projection $M\to M_1$ induces a morphism
  \[  p\colon\ \ \wt{\sX}\ \to\ \sX\ ,\]
  where $\sX$ is a family of smooth Fano fourfolds in $M_1$, and $p$ is the blow-up of some subfamily $\sY\subset\sX$ of smooth surfaces, whose fibers have trivial Chow groups. Assume moreover the following:

\begin{enumerate}

\item the $M_i$ have trivial Chow groups, and $\CH^\ast(M_i)$ is generated by intersections of divisors, for $i=1,2$;

\item the family $\sX\to B$ has a generically defined MCK decomposition;

\item  
   \[  \iota_\ast c_1(N_{Y/X}) \ \ \in\ \CH^1(X)\cdot \CH^2(X)\ \ \subset\ \CH^3(X)\  ,\]
    for every fiber $X$ with subvariety $Y\subset X$, where $\iota\colon Y\to X$ denotes the inclusion.
    \end{enumerate}
  
  Then $\wt{\sX}\to B$ and $\sX\to B$ have the Franchetta property.
\end{lm}  
  
\begin{proof}
It will suffice to establish Franchetta for $\sX\to B$; Lemma \ref{bufrank} then implies Franchetta for $\wt{\sX}\to B$. Let $X$ be a fiber and $f\colon\wt{X}\to X$ the blow-up morphism. Because $X$ is a Fano fourfold with $H^3(X,\Q)=0$, we only need to check the Franchetta property in codimension 3 (cf. Lemma \ref{only} below).
We observe that $f_\ast f^\ast=\id$ and $p_\ast p^\ast=\id$ and so the generically defined cycles on $X$ come from generically defined cycles on 
$\wt{X}$, i.e.
  \[  \GDCH^3_B(X)=  f_\ast  \GDCH^3_B(\wt{X})\ .\]

  Since the family $\tilde{\sX}$ is obtained as smooth zero locus of sections of the globally generated vector bundle $F$ on $M=M_1\times M_2$, applying \cite[Proposition 2.6]{FLV3} (with $r=1$)
  we obtain the equality

$$\GDCH^*_B(\wt{X})= \Im \Bigl(\CH^*(M) \to \CH^*(\wt{X})\Bigr)\ .$$
  
  Since  $\CH^*(M)$ is generated by divisors, it follows that
  \[ \GDCH^3_B(\wt{X})\ \subset\      \GDCH^1_B(\wt{X})\cdot  \GDCH^1_B(\wt{X})\cdot  \GDCH^1_B(\wt{X})\ .\]
   But the blow-up formula tells us that $\CH^1(\wt{X})= f^\ast \CH^1(X)\oplus\Q[E]$, where $E\subset\wt{X}$ is the exceptional divisor, and so
  \[   \GDCH^3_B(\wt{X})\ \subset\ \CH^2(\wt{X})\cdot f^\ast \CH^1({X}) + \Q [E^3]\ .\]  
  Pushing forward to $X$ and applying the projection formula, we thus find that
  \[    \GDCH^3_B(X)=  f_\ast  \GDCH^3_B(\wt{X})\ \subset\ \CH^2(X)\cdot \CH^1(X) + \Q  [ f_\ast (E^3)]\ .  \]
  Applying Sublemma \ref{21} below, the first summand is contained in $\CH^3_{(0)}(X)$. As for the second summand, we have equality
  \[   f_\ast (E^3) =  -\iota_\ast  c_1(N_{Y/X})  \ \ \hbox{in}\ \CH^3(X)\  \]
\cite[Example 3.3.4]{F}, and so the hypothesis plus Sublemma \ref{21} implies that the second summand is also contained in $\CH^3_{(0)}(X)$. 
It thus follows that
  \[  \GDCH^3_B(X)\ \ \subset\ \CH^3_{(0)}(X)\ .\]
  Since $\CH^3_{(0)}(X)$ is known to inject into cohomology (Lemma \ref{injection}), this shows the Franchetta property for  $\sX\to B$, and closes the proof.

\begin{sublm}\label{21} Let $X$ be a smooth Fano fourfold with an MCK decomposition, and $H^3(X,\Q)=0$. Then
  \[  \CH^2(X)\cdot \CH^1(X)\ \ \subset\ \CH^3_{(0)}(X)\ .\]
\end{sublm}

To prove the sublemma, 
it suffices to observe that $\CH^i(X)=\CH^i_{(0)}(X)$ for $i=1,2$ (cf. Lemma \ref{only} below).

(NB: when $X$ is a cubic fourfold, the sublemma also follows more directly from the excess intersection formula: $\CH^1(X)$ is generated by the hyperplane class $h$, and
  \[  \CH^2(X)\cdot h \ \subset\ \iota^\ast \iota_\ast \CH^2(X)\ \subset\ \iota^\ast \CH^3(\P^5)=\Q[h^3]\ \ =\ \CH^3_{(0)}(X)\ ,\]
  where $\iota\colon X\to\P^5$ denotes the inclusion morphism.)
   \end{proof}

In the rest of this section, $B$ will denote the basis as defined in Lemma \ref{bufrank1}.

\subsection{C-1}

The first family, dubbed C-1 in \cite{B+}, is given by zero loci of a section of $\sO(0,3)\oplus \sO(1,1)$ on $\P^1\times \P^5$. These fourfolds are indeed described as the blow-up of a general cubic fourfold along a cubic surface.

\smallskip

In order to show the Franchetta property for the family C-1, it is enough to remark that each cubic surface $Y$ inside a cubic fourfold $X$ is a linear section given by a $\P^3\subset \P^5$. In fact the normal bundle of $Y$ is the restriction to $Y$ of the normal bundle of $\P^3$ in $\P^5$. This in turn implies that $c_1(N_{Y/X})=h_{|Y}$. It follows that $\iota_*c_1(N_{Y/X})= h^3 \in \CH^3(X)$, and the hypotheses of Lemma \ref{bufrank1} are verified (we recall that cubic fourfolds have an MCK decomposition by Theorem \ref{mckcubic}). Since the Franchetta property holds, Proposition \ref{busurface} gives us the claimed MCK decomposition for the family C-1.

\subsection{C-3}

The second family, dubbed C-3 in \cite{B+}, is given by zero loci of a section of $\sO(1,2)\oplus \sQ_{\P^2}(0,1)$ on $\P^2\times \P^5$. These fourfolds are also described as the blow-up of a cubic fourfold in $\sC_8$ along a plane.

\smallskip

Since Fano fourfolds from family C-3 are blow-ups of cubic fourfolds $X$ along a plane $Y$, and $\CH^1(Y)=\Q$, we clearly have that
  \[ \iota_\ast \CH^1(Y)= [Y]\cdot h\ \ \subset\ \CH^2(X)\cdot \CH^1(X)\ .\]
We can thus apply Lemma \ref{bufrank1} and obtain the Franchetta property for $\sX$. The rest of the proof is the same as for family C-1.

\subsection{C-10}

As explained in \cite{B+}, Fano fourfolds in the family C-10 can be described as blow-ups $\wt{X}$ of $X$ with center $Y$, where $X$ is a Fano variety in the family C-1 and $Y$ is the blow-up of a cubic surface in 3 points (obtained by intersecting the cubic surface in $\P^5$ with a general $\P^3$). 

\smallskip

It is readily seen that the normal bundle of $Y$ in $X$ is the restriction of the normal bundle of a blown-up $\P^3$ in a blown-up $\P^5$, and so 
  \[ c_1(N_{Y/X})\ \ \in\  \ima\bigl( \CH^1(X)\to \CH^1(Y)\bigr)\ .\]
  The hypotheses of Lemma \ref{bufrank1} are then satisfied, and so the family $\sX\to B$ has the Franchetta property. Proposition \ref{busurface} now gives us the claimed MCK decomposition for the family C-10.
  
\subsection{C-12}  
  
 The family C-12 is treated in a similar fashion: as shown in \cite{B+}, Fano fourfolds in the family C-12 can be obtained as blow-ups $\wt{X}$ of $X$ with center $Y$, where $X$ is a special Fano variety in the family C-1 and $Y$ is the strict transform of a plane. 

\smallskip 
 
 Since the normal bundle of $Y$ in $X$ can be described as a restriction, the same reasoning as above readily gives that
    \[ c_1(N_{Y/X})\ \ \in\  \ima\bigl( \CH^1(X)\to \CH^1(Y)\bigr)\ .\]   
 The hypotheses of Lemma \ref{bufrank1} are thus again satisfied, and so the family $\sX\to B$ has the Franchetta property. Proposition \ref{busurface} 
 then gives us an MCK decomposition for the family C-12.
 
\subsection{C-15} 
 
Fano fourfolds in the family C-15 can be obtained as blow-ups $\wt{X}$ of $X$ with center $Y$, where $X$ is a special Fano variety in the family C-3 and $Y$ is a plane.

\smallskip

 Since $\CH^1(Y)=\Q$ the conditions of Lemma \ref{bufrank1} are once more satisfied, and the family $\sX\to B$ has the Franchetta property. Proposition \ref{busurface} 
 then gives an MCK decomposition for the family C-15.

\subsection{C-6}

Fano fourfolds from family C-6 of \cite{B+} have a description as one-nodal cubic fourfolds, blown-up in the node. By \cite[Corollary 5.6]{55}, the blow-up of a one-nodal cubic hypersurface (in any even dimension) with center the node has an MCK decomposition.

\subsection{Families of type K3-n}

In the following subsections we will show the existence of MCK decompositions for certain families of Fano fourfolds of type K3-n. These Fano fourfolds are obtained as blow-ups of Fano varieties with trivial Chow groups, with centers birational to K3 surfaces. First of all, we need a technical result.

\begin{prop}\label{frk3}
Let $\sX,\sY$ and $\wt{\sX}$ be families (respectively of fiber dimension 4, 2 and 4) of projective varieties obtained by proceeding as in Lemma \ref{cayley} over the base

$$B\subset \overline{B}:=\P H^0(\P^1\times M_2, \sO_{\P^1}(1)\boxtimes L \oplus \sO_{\P^1}\boxtimes F).$$

Assume that we have an injection

$$\Im\bigl(\CH^2(M_2)\to\CH^2(Y_b)\bigr)\ \hookrightarrow\  H^4(Y_b,\Q)\ ,$$

for any fiber $Y_b$. Then $\sY\to B$ has the Franchetta property.

\end{prop}

\begin{proof}
We will denote as usual by $Y$ the surface contained in $X$ and by $\wt{X}\to X$ the blow-up with center $Y$, and exceptional divisor $E$. The natural map $E\to Y$ will be denoted by $g$, and $\xi$ will be the class of the tautological sheaf on $E$.

We need to show that

\begin{equation}\label{needto}
\GDCH^2_B(Y)=\Im\bigl(\CH^2(M_2) \to CH^2(Y)\bigr)\ .
\end{equation}

Since our vector bundle is globally generated, it follows from \cite[Proposition 2.6]{FLV3} that

$$\GDCH^*_B(\wt{X})= \Im\bigl(\CH^*(\P^1\times M_2) \to \CH^*(\wt{X})\bigr)\ ,$$

and so, in particular

$$\Im\bigl(\GDCH^*_B(\wt{X}) \to \GDCH^*_B(E)\bigr) = \Im\bigl(\CH^*(\P^1\times M_2) \to \CH^*(E)\bigr)\ .$$

Then one sees that

$$g^*\GDCH^2(Y)\cdot \xi \subset \Im\bigl(\CH^3(\P^1 \times M_2) \to \CH^3(E)\bigr)\ .$$

Moreover $\CH^3(\P^1\times M_2)= \CH^3(M_2)\oplus \CH^2(M_2)\cdot h$, where $h$ is the class of the tautological bundle of the projective bundle $\P^1\times M_2 \to M_2$. \footnote{Remark that the class $\xi$ is just the restriction of $h$ to $E$} On the other hand, we have $\CH^3(E)=\CH^2(Y)\cdot\xi$. We want to show that $g^*\GDCH^2(Y)\cdot \xi$ is indeed contained in $\Im(\CH^2(M_2)\cdot h\to \CH^3(E))$. It is straightforward to see that 

$$\Im\bigl(\CH^3(M_2)\to \CH^3(E)\bigr)=0\ ,$$
since $\CH^3(Y)=0$ and the map factors through $\CH^3(Y)$. Consider now an element $g^*(a)\cdot h\in \GDCH^2(Y)\cdot \xi$. Let $p_2: \P^1\times M_2 \to M_2$ be the natural projection. Then we have that 

$$g^*(a)\cdot h\ \in\ \Im\bigl(p_2^*\CH^2(M_2)\cdot h \to g^*\CH^2(Y)\cdot \xi\bigr)\ ,$$
and functoriality of pull-back gives the equality

$$\Im\bigl(p_2^*\CH^2(M_2)\cdot h \to g^*\CH^2(Y)\cdot \xi\bigr) = g^*\Im\bigl(\CH^2(M_2) \to \CH^2(Y)\bigr)\cdot h\ .$$
Combining the last two inclusions, we find that 
  \[  g^*(a)\cdot h\ \in\ g^*\Im\bigl(\CH^2(M_2) \to \CH^2(Y)\bigr)\cdot h\ .\] 
Finally, by applying the projection formula we find that $a$ lies in $\Im\bigl(\CH^2(M_2)\to\CH^2(Y)\bigr)$. This proves equality \eqref{needto}.
\end{proof}

\subsection{K3-24} A Fano fourfold $\wt{X}$ of type K3-24 is obtained as the blow-up of $X=X_{12}$ with center a genus 6 K3 surface $Y\subset\Gr(2,5)$.
Here $M_2=\Gr(2,5)$ and the injection \eqref{injpsy} holds for $g=6$ thanks to \cite{PSY}. Proposition \ref{frk3} then gives the Franchetta property for the family $\sY\to B$. The Fano fourfold $X=X_{12}$ has trivial Chow groups (this is proven directly in \cite{JLMS}; alternatively this follows from the existence of a full exceptional collection for the derived category of $X$ \cite{Kuz}). 

Applying Proposition \ref{busurface} we obtain an MCK decomposition (that is generically defined) for Fano fourfolds of type K3-24.

\subsection{K3-25}
A Fano fourfold $\wt{X}$ of type K3-25 is obtained as the blow up of $X=X_{14}$ along a genus 8 K3 surface $Y\subset \Gr(2,6)$. In this case $M_2=\Gr(2,6)$. The fourfold $X_{14}$ has trivial Chow groups \cite{JLMS}, and the family of degree 14, genus 8, K3 surfaces has the Franchetta property by Proposition \ref{frk3}, since injection \eqref{injpsy}  holds for $g=8$ by the work of \cite{PSY}. This implies as above the existence of a generically defined MCK decomposition.

\subsection{K3-26}
A Fano fourfold $\wt{X}$ of type K3-26 is obtained as the blow up of $X=X_{16}$ (a codimension two linear section of the Lagrangian Grassmannian $LG(3,6)$) along a genus 9 K3 surface $Y\subset LG(3,6)$. The fourfold $X_{16}$ has trivial Chow groups since its derived category has a full exceptional collection \cite[Section 6.3]{kuzhypsec}. The injection \eqref{injpsy}  for $g=9$ holds \cite{PSY} and so the Franchetta property holds true for the family of K3 surfaces. Thus we obtain a generically defined MCK decomposition.

\subsection{K3-28}
Fourfolds from this family are blow-ups of Fano fourfolds $X=X_{18}$ along a genus 10 K3 surface. Here $M_2=\Gr(2,7)$.
The argument goes exactly along the same lines, following the same references, this time for $g=10$, as in case K3-24
(for the existence of the full exceptional collection see \cite[Section 6.4]{kuzhypsec}).

\subsection{K3-30} A Fano fourfold $\wt{X}$ of type K3-30 is the blow-up of $X=\Gr(2,4)$ with center a genus 5 K3 surface $Y\subset X$. It is known that
  \[ \ima\bigl( \CH^2(\Gr(2,4)\to \CH^2(Y)\bigr)=\Q \]
  (this follows from \cite[First proof of Proposition 2.1]{PSY}), and so Proposition \ref{frk3} guarantees the Franchetta property for $\sY\to B$. As clearly $X$ has trivial Chow groups, this gives the MCK decomposition for $\wt{X}$.



\subsection{K3-40}
Here $M_2=\P^1\times \P^5$ and its Chow ring is generated by intersections of divisors. In view of Proposition \ref{frk3}, this directly implies the Franchetta property for the family $\sY\to B$, whose members here are K3 surfaces of degree 8, blown up in 8 points. The fourfold $X$ is the blow-up of a 4-dimensional quadric along a 2-dimensional quadric, hence clearly has trivial Chow groups. Hence, we obtain a generically defined MCK decomposition for $\wt{X}$.

\subsection{K3-41}
For this family $M_2=\P^1\times \P^3$, and $X$ coincides with $M_2$. The fourfold $X$ has trivial Chow groups and its Chow ring is generated by intersections of divisors. This implies the Franchetta property for the family $\sY\to B$ of K3 surfaces, that in this case are bielliptic. The argument is the same as the preceding one.

\subsection{K3-46}This case is similar to the two preceding ones. The K3 surfaces $\sY\to B$ are some special degree 26 surfaces (see \cite{B+} for more details). In this case we have $M_2=\P^3\times \P^5$, hence its Chow ring is generated by intersections of divisors and the Franchetta property follows for $\sY\to B$. On the other hand $X$ is the blow up of a smooth quadric along a line, hence has trivial Chow groups. One concludes as in K3-40 and K3-41.

\subsection{K3-47}
For Fano fourfolds of type K3-47, the family of surfaces is made up by determinantal quartic K3 surfaces. We have $M_2=\P^3\times \P^3$, and its Chow ring is thus generated by intersections of divisors. As before this implies the Franchetta property for the K3 family $\sY\to B$. In this case $X$ is a codimension 2 linear section of $\P^3\times \P^3$, which has trivial Chow groups since it has a full exceptional collection by an application of \cite{hpdsegre}. The upshot is that the fourfold $\wt{X}$ has a generically defined MCK decomposition.

\subsection{K3-49}
In this case the family $\sY\to B$ is made up of degree 20 K3 surfaces. Moreover $M_2=\P^2\times \P^4$, whose Chow ring is generated by intersections of divisors. This implies the Franchetta property for $\sY$. The fourfold $X$ is a $(1,1)$-divisor of $\P^2\times Q^3$ (where $Q^3$ is a smooth quadric threefold), that has a full exceptional collection by the description in \cite{B+}. Hence it also has trivial Chow groups. This implies the existence of the  MCK decomposition.

\subsection{K3-55} The family of surfaces $\sY\to B$ here is made up of degree 30 K3 surfaces.
We have $M_2=\P^4\times \P^5$, thus once again the intersections of divisors generate the Chow ring, and the Franchetta property holds for K3 surfaces. Furthermore, the variety $X$ can be seen as a $\P^1$-bundle on a smooth quadric threefold, hence has trivial Chow groups, and so the blow-up $\wt{X}$ has a generically defined MCK decomposition.

\subsection{K3-56} For this family, the K3 surfaces are certain degree 28 K3 surfaces, with Picard rank 3.
The Chow ring of $M_2=\P^1\times \P^2 \times \P^2$ is generated by intersections of divisors, which implies the Franchetta property for $\sY\to B$. The variety $X$ is isomorphic to the product of $\P^1$ and the blow-up of $\P^3$ in one point,  hence has trivial Chow groups. This implies that $\wt{X}$ has a generically defined MCK decomposition.

\subsection{K3-58} In this case, the family of surfaces $\sY\to B$ is made up of bielliptic K3 surfaces.
On the other hand, $M_2=\P^1\times \P^1\times \P^3$ and its Chow ring is generated by divisors. This implies the Franchetta property for $\sY\to B$. Now, $X$ is the product of $\P^1$ and the blow-up of $\P^3$ along a line, thus it has trivial Chow groups. This implies that $\wt{X}$ has an MCK decomposition.

\subsection{K3-59} The family of surfaces for K3-59 is given by K3 surfaces that are isomorphic to a degree 20 K3.
Here $M_2$ is $\P^1\times \P^2\times \P^2$, and its Chow ring is generated by intersections of divisors. This in turn implies the Franchetta property for $\sY\to B$. On the other hand, $X$ is the product of $\P^1$ and a flag manifold. The upshot is that $X$ has trivial Chow groups. The combination of these results implies that there is an MCK decomposition for $\wt{X}$.

\subsection{K3-60} The K3 surfaces for this family are degree 24 surfaces with Picard rank 4.
In this case $M_2$ coincides with $X=(\P^1)^4$, and it has both trivial Chow groups and Chow ring generated by intersections of divisors. The latter property implies the Franchetta property for $\sY\to B$, and, combining this with the former, one obtains that $\wt{X}$ has a generically defined MCK decomposition.




\begin{rk} The families K3-30 and K3-41 already appear in \cite{FM}, under the labels $B_2$ resp. $B_1$. An MCK decomposition for these families was constructed in \cite{40}, by arguments similar to those of the present paper.
\end{rk}

\section{A consequence for the intersection product}

Our Theorem \ref{main} has some interesting consequences on the behaviour of the intersection product with respect to the cycle class map. Not surprisingly, Fano fourfolds that appear in Theorem \ref{main} behave very much like K3 surfaces.

%

\begin{cor}\label{mainc}

 Let $X$ be one of the Fano fourfolds in Table 1. Then the image of the intersection product map
 
  \[ m: \CH^1(X)\otimes \CH^2(X)\ \to\  \CH^3(X) \] 
  injects into cohomology. In other words, there are $\rho:=\dim H^2(X,\Q)$ distinguished 1-cycles $\ell_1,\ldots,\ell_\rho$ such that
  
  \[  \Im\Bigl(  \CH^1(X)\otimes \CH^2(X)\ \to\  \CH^3(X)\Bigr) =\bigoplus_{i=1}^\rho \Q[\ell_i]\ .\]

\end{cor}

\begin{proof}

Note that for any Fano fourfold $X$ in Table 1 we have $H^3(X,\Q)=0$ (indeed, these Fano fourfolds are obtained either as blow-up of a fourfold with trivial Chow groups with center a K3 surface, or as repeated blow-up of a cubic fourfold with center some surfaces with trivial Chow groups).

In view of Lemma \ref{only} below, we thus have
  \[ \CH^i(X)=\CH^i_{(0)}(X)\ \ \forall\ i\not=3\ .\]
  As $\CH^\ast_{(\ast)}(X)$ is a bigraded ring under the intersection product, the image of the map $m$ is thus contained in $\CH^3_{(0)}(X)$.
But $\CH^3_{(0)}(X)$ injects into cohomology by virtue of Lemma \ref{injection2}. 

The rest of the claim is a straightforward cohomology computation.

\begin{lm}\label{only} Let $X$ be a smooth Fano fourfold with $H^3(X,\Q)=0$. Then
  \[ \CH^i_{hom}(X)=0\ \ \forall\ i\not=3\ .\]
  In particular, if in addition $X$ has an MCK decomposition then
  \[ \CH^i(X)=\CH^i_{(0)}(X)\ \ \forall\ i\not=3\ .\]
\end{lm}

To prove the lemma, we recall that smooth Fano varieties $X$ (are rationally connected and so) have $\CH_0(X)=\Q$. The Bloch--Srinivas argument \cite{BS} then implies that the Abel--Jacobi map
  \[ \CH^2_{hom}(X)\ \to\ \JJ^3(X) \]
  is injective, where $\JJ^3(X)$ is the intermediate Jacobian. The assumption on $H^3(X)$ guarantees that $\JJ^3(X)=0$, and so $\CH^2_{hom}(X)=0$. This proves the first statement.
  
  For the second statement of the lemma, this follows from the first statement upon observing that for any variety $X$ with an MCK decomposition one has
  \[  \CH^i_{(j)}(X)\ \subset\ \CH^i_{hom}(X)\ \ \forall\ j\not=0\ .\]
  This closes the proof of the lemma and of the corollary.
\end{proof}

\vskip1.0cm

\begin{nonumberingt} M.B. wants to thank R.L. for an invitation in April 2022 to Strasbourg, where this paper was started. R.L. wants to thank M.B. back for an invitation in September 2022 to Montpellier, where this paper was completed. Both of us thank Marcello Bernardara, Enrico Fatighenti, Laurent Manivel and Fabio Tanturri for kindly answering our questions. Finally, thanks to the referee for insightful comments that significantly improved our paper.
\end{nonumberingt}


\vskip1.0cm
\begin{thebibliography}{dlPG99}


	
	
		

\bibitem{Beau3} A. Beauville, On the splitting of the Bloch--Beilinson filtration, in: Algebraic cycles and motives (J. Nagel and C. Peters, editors), London Math. Soc. Lecture Notes 344, Cambridge University Press 2007,

\bibitem{BV} A. Beauville and C. Voisin, On the Chow ring of a K3 surface, J. Alg. Geom. 13 (2004), 417---426,

\bibitem{BL} N. Bergeron and Z. Li, {Tautological classes on moduli space of hyperk\"ahler manifolds}, Duke Math. J., arXiv:1703.04733,

\bibitem{BS} S. Bloch and V. Srinivas, Remarks on correspondences and algebraic cycles, American Journal of Mathematics Vol. 105 no. 5 (1983), 1235--1253,

\bibitem{BL22} M. Bolognesi and R. Laterveer, Some motivic properties of Gushel-Mukai sixfolds, arXiv:2201.11175, 22 pages, (2022),


\bibitem{B+} M. Bernardara, E. Fatighenti, L. Manivel and F. Tanturri, Fano fourfolds of K3 type, arXiv:2111.13030,

\bibitem{BP} M. Bolognesi and C. Pedrini, 
The transcendental motive of a cubic fourfold, J. Pure Appl. Algebra 224 (2020), no. 8, 106333, 16 pp.,

\bibitem{hpdsegre} M. Bernardara, M. Bolognesi, D. Faenzi, 
Homological projective duality for determinantal varieties,
Adv. Math. 296 (2016), 181---209,

\bibitem{Bu} T.B\"ulles,
Motives of moduli spaces on K3 surfaces and of special cubic fourfolds,
Manuscripta Math. 161 (2020), no. 1-2, 109---124.







\bibitem{Diaz} H. Diaz, The Chow ring of a cubic hypersurface, to appear in International Math. Research Notices,



\bibitem{FM} E. Fatighenti and G. Mongardi, Fano varieties of K3 type and IHS manifolds,  Int. Math. Res. Not. (IMRN), Volume 2021, Issue 4, 3097--3142
	
\bibitem{FLV} L. Fu, R. Laterveer and Ch. Vial, The generalized Franchetta conjecture for some hyper-K\"ahler varieties (with an appendix joint with M.
Shen), Journal Math. Pures et Appliqu\'ees (9) 130 (2019), 1---35,

\bibitem{FLV3} L. Fu, R. Laterveer and Ch. Vial, The generalized Franchetta conjecture for some hyper-K\"ahler varieties, II, Journal de l'Ecole Polytechnique--Math\'ematiques 8 (2021), 1065---1097,

\bibitem{FLV2} L. Fu, R. Laterveer and Ch. Vial, Multiplicative Chow--K\"unneth decompositions and varieties of cohomological K3 type, Annali Mat. Pura ed Applicata 200 (2021), 2085---2126,

\bibitem{FL} L. Fu and R. Laterveer, Special cubic fourfolds, K3 surfaces and the Franchetta property, Int. Math. Research Notices, https://doi.org/10.1093/imrn/rnac102,

\bibitem{FTV} L. Fu, Z. Tian and Ch. Vial, Motivic hyperk\"ahler resolution conjecture for generalized Kummer varieties, Geometry \& Topology 23 (2019), 427---492,

\bibitem{FV} L. Fu and Ch. Vial, Distinguished cycles on varieties with motive of abelian type and the section property, J. Alg. Geom. 29 (2020), 53---107,

\bibitem{FV2} L. Fu and Ch. Vial, A motivic global Torelli theorem for isogenous K3 surfaces, Advances in Mathematics 383 (2021),

\bibitem{F} W. Fulton, Intersection theory, Springer--Verlag Ergebnisse der Mathematik, Berlin Heidelberg New York Tokyo 1984,


\bibitem{J4} U. Jannsen, On finite-dimensional motives and Murre's conjecture, in: Algebraic cycles and motives (J. Nagel and C. Peters, editors), Cambridge University Press, Cambridge 2007,

\bibitem{Kim2} S.-I. Kimura, Surjectivity of the cycle map for Chow motives, in: Motives and algebraic cycles, Fields Inst. Commun., vol. 56, Amer. Math. Soc., Providence, RI, 2009, pp. 157---165,


\bibitem{Kuz} A. Kuznetsov, Homological projective duality for Grassmannians of lines, math.AG/0610957,

\bibitem{kuzhypsec} A. Kuznetsov, Hyperplane sections and derived categories. Izv. Ross. Akad. Nauk Ser. Mat. 70 (2006), no. 3, 23---128; translation in
Izv. Math. 70 (2006), no. 3, 447---547,




\bibitem{23.6} R. Laterveer, Algebraic cycles on some special hyperk\"ahler varieties, Rendiconti di Matematica e delle sue applicazioni 38 no. 2 (2017), 243---276, 

\bibitem{37} R. Laterveer, A remark on the Chow ring of K\"uchle fourfolds of type $d3$, Bulletin Australian Math. Soc. 100 no. 3 (2019), 410---418,

\bibitem{38} R. Laterveer and Ch. Vial, On the Chow ring of Cynk--Hulek Calabi--Yau varieties and Schreieder varieties, Canadian Journal of Math. 72 no. 2 (2020), 505---536,

\bibitem{39} R. Laterveer, Algebraic cycles and Verra fourfolds, Tohoku Math. J. 72 no. 3 (2020), 451---485,

\bibitem{40} R. Laterveer, On the Chow ring of certain Fano fourfolds, Ann. Univ. Paedagog. Crac. Stud. Math. 19 (2020), 39---52,

\bibitem{41} R. Laterveer, On the Chow ring of Fano varieties of type $S6$, Serdica Math. J. 45 (2019), 289---304,

\bibitem{44} R. Laterveer, On the Chow ring of Fano varieties of type $S2$, Abh. Math. Semin. Univ. Hambg. 90 (2020), 17---28,

\bibitem{45} R. Laterveer, Algebraic cycles and special Horikawa surfaces, Acta Math. Vietnamica 46 (2021), 483---497,

\bibitem{46} R. Laterveer, Algebraic cycles and Gushel--Mukai fivefolds, Journal of Pure and Applied Algebra 225 no. 5 (2021), 

\bibitem{48} R. Laterveer, Algebraic cycles and intersections of 2 quadrics, Mediterranean Journal of Mathematics 18 no. 4 (2021), 
doi: 10.1007/s00009-021-01787-5,

\bibitem{55} R. Laterveer, Algebraic cycles and intersections of a quadric and a cubic, Forum Mathematicum  33 no. 3 (2021), 845---855, 

\bibitem{g8} R. Laterveer, Algebraic cycles and Fano threefolds of genus 8, Portugaliae Math. (N.S.) Vol. 78, Fasc. 3-4 (2021), 255---280,

\bibitem{JLMS} R. Laterveer, Motives and the Pfaffian--Grassmannian equivalence, Journal of the London Math. Soc. 104 no. 4 (2021), 1738---1764,

\bibitem{59} R. Laterveer, On the Chow ring of Fano varieties on the Fatighenti-Mongardi list, Communications in Algebra 50 no. 1 (2022), 131---145,

\bibitem{60} R. Laterveer, Algebraic cycles and intersections of three quadrics, Mathematical Proceedings Cambridge Philosophical Society 173 no. 2 (2022), 349---367,

\bibitem{M1} S. Mukai, Curves, K3 surfaces and Fano 3-folds of genus $\le$ 10, in: Algebraic geometry and commutative algebra, Vol. I, 357---377, Kinokuniya, Tokyo, 1988,

\bibitem{M2} S. Mukai, Polarized K3 surfaces of genus 18 and 20, in: Complex projective geometry (Trieste, 1989/Bergen, 1989), 264---276, London Math. Soc. Lecture Note Ser., 179, Cambridge Univ. Press, Cambridge, 1992,

\bibitem{M3} S. Mukai, Polarized K3 surfaces of genus thirteen, in: Moduli spaces and arithmetic geometry, 315---326, Adv. Stud. Pure Math., 45, Math. Soc. Japan, Tokyo, 2006,

\bibitem{M4} S. Mukai, K3 surfaces of genus sixteen. Preprint, 2012, available at http://www.kurims. kyoto-u.ac.jp/preprint/file/RIMS1743.pdf,

\bibitem{Mur} J. Murre, On a conjectural filtration on the Chow groups of an algebraic variety, parts I and II, Indag. Math. 4 (1993), 177---201,

\bibitem{MNP} J. Murre, J. Nagel and C. Peters, Lectures on the theory of pure motives, Amer. Math. Soc. University Lecture Series 61, Providence 2013,

		
\bibitem{PSY} N. Pavic, J. Shen and Q. Yin, On O'Grady's generalized Franchetta conjecture, Int. Math. Res. Notices (2016), 1---13,
		

\bibitem{Sc} T. Scholl, Classical motives, in: Motives (U. Jannsen et alii, eds.), Proceedings of Symposia in Pure Mathematics Vol. 55 (1994), Part 1, 

\bibitem{SV} M. Shen and Ch. Vial, The Fourier transform for certain hyperK\"ahler fourfolds, Memoirs of the AMS 240 (2016), no.1139,

\bibitem{SV2} M. Shen and Ch. Vial, The motive of the Hilbert cube $X^{[3]}$, Forum Math. Sigma 4 (2016), 55 pp., 


2040---2062,

\bibitem{V10} Ch. Vial, Pure motives with representable Chow groups, C. R. Math. Acad. Sci. Paris 348 (2010), no. 21-22, 1191---1195, 

\bibitem{V6} Ch. Vial, On the motive of some hyperk\"ahler varieties, J. Reine Angew. Math. 725 (2017), 235---247,


\bibitem{V0} C.Voisin,The generalized Hodge and Bloch conjectures are equivalent for general complete intersections, Ann. Sci. Ecole Norm. Sup. 46, fascicule 3 (2013), 449---475,


	

\end{thebibliography}
\end{document}